\documentclass[12pt,a4paper]{article}
\usepackage{afterpage}
\usepackage{fancyhdr}
\usepackage{calc}
\usepackage{color}
\usepackage{amsmath,amsthm,amssymb}
\usepackage{cases}
\usepackage{graphicx}
\usepackage{float}
\usepackage[numbers]{natbib}

\usepackage[margin=2em,labelfont=bf]{caption}
\usepackage[english]{babel}
\newtheorem{thm}{Theorem}

\newtheorem{ass}[thm]{Assumption}

\definecolor{DarkBlue}{rgb}{0,0,0.5451}
\definecolor{DarkGreen}{rgb}{0,0.39216,0}
\definecolor{LightYellow}{rgb}{1,1,.8}
\definecolor{orange}{rgb}{.9,0.3445,0}

\makeatletter
\let\oldc\c

\let\oldnorm\|

\renewcommand{\|}{|\!|}         % closer norm

\newcounter{algo}

\DeclareMathOperator{\sinc}{sinc}

\newcommand{\cD}{{\cal{D}}}

\newcommand{\IC}{{\mathbb{C}}}

\newcommand{\IN}{{\mathbb{N}}}

\newcommand{\IR}{{\mathbb{R}}}

\newcommand{\btheta}{{\boldsymbol{\theta}}{}}

\newcommand{\bSigma}{{\boldsymbol{\Sigma}}{}}

\newcommand{\A}{{\mathbf{A}}}

\newcommand{\C}{{\mathbf{C}}}

\newcommand{\bfM}{{\mathbf{M}}}

\newcommand{\bfT}{{\mathbf{T}}} % \T transpose

\newcommand{\1}{{\mathbf{1}}}

\renewcommand{\c}{{\textbf{\textit{c}}}}
\newcommand{\f}{{\textbf{\textit{f}}}}

\newcommand{\h}{{\textbf{\textit{h}}}}
\newcommand{\s}{{\textbf{\textit{s}}}}

\renewcommand{\v}{{\textbf{\textit{v}}}}
\newcommand{\w}{{\textbf{\textit{w}}}}
\newcommand{\x}{{\textbf{\textit{x}}}}

\newcommand{\z}{{\textbf{\textit{z}}}}

\newcommand*{\stack@relbin}[3][]{%
  \mathop{#3}\limits
  \toks@{#1}%
  \edef\reserved@a{\the\toks@}%
  \ifx\reserved@a\@empty\else_{#1}\fi
  \toks@{#2}%
  \edef\reserved@a{\the\toks@}%
  \ifx\reserved@a\@empty\else^{#2}\fi
  \egroup
}%
\renewcommand*{\stackrel}{%
  \mathrel\bgroup\stack@relbin
}
\newcommand*{\stackbin}{%
  \mathbin\bgroup\stack@relbin
}
\usepackage[numbers]{natbib}

\begin{document}

\title{On the smallest eigenvalues of covariance matrices of multivariate spatial processes}

\author{Fran\oldc{c}ois Bachoc, Reinhard Furrer \\
Toulouse Mathematics Institute, University Paul Sabatier, France \\
Institute of Mathematics and Institute of Computational Science, \\
 University of Zurich, Switzerland}

\maketitle

{\bf Abstract:}
There has been a growing interest in providing models for multivariate spatial processes. A majority of these models specify a parametric matrix covariance function. Based on observations, the
parameters are estimated by maximum likelihood or variants
thereof. While the asymptotic properties of maximum likelihood
estimators for univariate spatial processes have been analyzed in
detail, maximum likelihood estimators for multivariate spatial
processes have not received their deserved attention yet. In this article we consider the classical increasing-domain asymptotic setting restricting the minimum distance between the locations.
Then, one of the main components to be studied from a theoretical point of
view is the asymptotic positive definiteness of the underlying covariance matrix. Based on very weak assumptions on the matrix covariance function we show that the smallest eigenvalue of the covariance matrix is asymptotically bounded away from zero. Several practical implications are discussed as well.


\maketitle

\section{Introduction}
The motivation of this work is the question of what simple conditions
on the sampling design and covariance functions have to be imposed for
a particular likelihood estimator to be consistent or, say, asymptotically normal, in the setting of
multivariate spatial processes.  More precisely, let
\begin{align}
  \bigl\{ Z_k(\s): \s\in\cD\subset\IR^d,   1 \leq k \leq p  \bigr\}
\label{eq:process}
\end{align}
be a multivariate stationary random process, for some fixed
$d\in\IN^+$ and $p\in\IN^+$. Without loss of generality, we assume
that the process~\eqref{eq:process} has zero mean and a matrix
covariance function of the form $\C(\h)= \{c_{k\ell}(\h)\}_{1 \leq
  k,\ell \leq p}$.  The diagonal elements $c_{kk}$ are called
direct covariance (functions) and the off-diagonal entries
$c_{k\ell}$, $k\neq l$, are called cross covariance (functions).

Consider $p$ sets $(\x^{(1)}_i)_{1 \leq i \leq n_1},\dots,(\x^{(p)}_i)_{1 \leq i \leq n_p}$ of points in $\cD$. In the geostatistical
literature, we often denote these points as locations and assume
that each $Z_k(\s)$ is observed at $(\x^{(k)}_i)_{1 \leq i \leq n_k}$.
One important problem is to estimate the matrix covariance function from these observations. One common practice is to assume that this function belongs to a given parametric set, and to estimate the corresponding covariance parameters by maximum likelihood approaches. These approaches are now largely implementable in practice, with the availability of massive computing power. In contrast, finite sample
and asymptotic properties of these approaches for multivariate processes are still on the research agendas.

Currently, the majority of the asymptotic results are specific to the univariate case, where $d=1$ and $n_1 = n \to \infty$. While two types of asymptotic settings exist, fixed-domain and increasing-domain \citep{stein99interpolation}, we shall focus throughout this article on the increasing-domain setting. 
In a seminal article, \cite{mardia84maximum} give conditions which warrant consistency and asymptotic normality of the maximum likelihood estimator.
Additional results and conditions are provided in \cite{bachoc14asymptotic} for the maximum likelihood estimator and in \cite{shaby12tapered} for the tapered maximum likelihood estimator. A central quantity in these three references is the covariance matrix of the observation vector. In particular, a necessary condition for the results above is that the smallest eigenvalue of this covariance matrix be bounded away from zero as $n \to \infty$. This seemingly difficult-to-check condition is guaranteed by much simpler and non-restrictive conditions applying only on the covariance function, provided that there exists a fixed minimal distance between any two different observation points, as shown by \cite{bachoc14asymptotic}.

The picture is similar in the multivariate setting ($d >1$), but currently incomplete. The recent references \cite{Beli:etal:15} and \cite{furrer15asymptotic} provide asymptotic results for maximum likelihood approaches, and require that the smallest eigenvalue of the covariance matrix of the observations be bounded away from zero as $n \to \infty$. However, this condition has currently only been shown to hold for some specific covariance functions and regular grids, see \cite{Beli:etal:15}. 

In this article, we show that this condition holds in general, provided that the minimal distance between two different observation points of the same process - $\x^{(k)}_i$ and $\x^{(k)}_j$ for $i \neq j$ - is bounded away from zero as $n_1,...,n_p \to \infty$, and provided that the matrix covariance function satisfies certain non-restrictive conditions generalizing those of \cite{bachoc14asymptotic}. 
While the starting argument is similar as in Proposition D.4 of
\cite{bachoc14asymptotic} the proof here requires different and more
elaborate tools based on complex matrices and Fourier functions. Our approach allows a proof in both cases, collocated and  non-collocated observations. We also show that the lower bound we provide can be made uniform over parametric families of matrix covariance functions.
We perceive this article as the first
step towards a rigorous analysis of the asymptotic properties of various
maximum likelihood type estimators for multivariate spatial processes
in the framework of increasing-domain asymptotics.

The next section introduces the notation and states the main result.
Due to the interesting tools used in the proof, it is of interest
on its own and we have kept it in the main part of the article. Section \ref{section:conclusion}
concludes with some remarks.

\section{Notation and main result}

For $\x \in \IC^m$, we let $|\x| =
\max_{1 \leq i \leq m} |x_i|$, where $|z|$ is the modulus of a complex number $z$.
Recall that any Hermitian complex matrix $\bfM$ of size $m \times m$ has real eigenvalues that we write $\lambda_1( \bfM ) \leq \dots \leq \lambda_m( \bfM ) $. For a complex vector $\v$ of size $m \times 1$, we let $\v^* = \bar{\v}^t$ be the transpose of its conjugate vector and we let $||\v||^2 = \v^* \v$. The next assumption is satisfied for most standard covariance and cross covariance functions.

\begin{ass} \label{ass:decay:infinity}
There exists a finite fixed constant $A > 0$ and a fixed constant $\tau >0$ so that the functions $c_{k\ell}$ satisfy, for all
$\x \in \IR^d$,
\begin{equation} \label{eq: controleCov}
\left| c_{k\ell}(\x) \right| \leq \frac{A}{ 1+|\x|^{d+\tau} }.
\end{equation}
\end{ass}

We define the Fourier transform of a function $g: \IR^d \to \IR$ by $\hat{g}(\f) = (2 \pi)^{-d} \int_{\IR^d} g(\x) e^{- \imath \f \cdot \x } d\x$,
where $\imath^2 = -1$.
Then, from \eqref{eq: controleCov}, the covariance functions $c_{k\ell}$ have Fourier transforms $\hat{c}_{k\ell}$ that are continuous and bounded. Also, note that, for any $\f \in \IR^d$, $\hat{\C}(\f) = \{ \hat{c}_{k\ell}(\f)\}_{1 \leq k,\ell \leq p}$ is a Hermitian complex matrix, that have real non-negative eigenvalues $0 \leq \lambda_1\{ \hat{\C}(\f)\} \leq \dots \leq \lambda_p\{ \hat{\C}(\f)\} $. We further assume the following.

\begin{ass} \label{ass:fourier}
We have 
\begin{equation} \label{eq:inverse:fourier}
c_{k\ell}(\x) = \int_{\IR^d} \hat{c}_{k\ell}(\f)  e^{ \imath \f \cdot \x } d\f.
\end{equation}
Also, we have $0 < \lambda_1\{ \hat{\C}(\f)\}$ for all $\f \in \IR^d$.
\end{ass}

The condition \eqref{eq:inverse:fourier} is very weak and satisfied by most standard covariance and cross covariance functions. On the other hand, the condition $0 < \lambda_1\{ \hat{\C}(\f)\}$ for all $\f \in \IR^d$ is less innocuous, and is further discussed in Section \ref{section:conclusion}.

Let $d \in \IN^+$ and $p \in \IN^+$ be fixed.
Consider $p$ sequences $(\x^{(1)}_i)_{i \in \IN^+},\dots,(\x^{(p)}_i)_{i \in \IN^+}$ of points in $\IR^d$, for which we assume the following.

\begin{ass} \label{ass:Delta}
There exists a fixed $\Delta >0$
so that for all $k$, $\inf_{i,j \in \IN^+;i \neq j} | \x^{(k)}_i - \x^{(k)}_j | \geq \Delta$.
\end{ass}

For all $n_1,\dots,n_p \in \IN^+$, let, for $0 \leq k \leq p$, $N_k = n_1+\dots+n_k$, with the convention that $N_0 = 0$. Let also $N = N_p$. Note that $N_1,\dots,N_p$ depend on $n_1,\dots,n_p$ but that we do not explicitly write this dependence for concision. Then, let $\bSigma$ (also depending on $n_1,\dots,n_p$) be the $N \times N$ covariance
matrix, filled as follows: For $a = N_{k-1} + i$ and $b = N_{\ell-1} + j$,
with $1 \leq k,\ell \leq p$, $1 \leq i \leq n_k$ and $1 \leq j \leq n_l$, $\sigma_{ a b} =
c_{k\ell}(\x^{(k)}_i-\x^{(\ell)}_j)$. 

Our main result is the following.

\begin{thm} \label{thm}
Assume that Assumptions \ref{ass:decay:infinity}, \ref{ass:fourier} and \ref{ass:Delta} are satisfied. Then, we have
\[
\inf_{n_1,\dots,n_p \in \IN^+} \lambda_1(\bSigma) > 0.
\]
\end{thm}

\begin{proof}[Proof of Theorem \ref{thm}]
From the proof of Proposition D.4 in \cite{bachoc14asymptotic}, there exists a function $\hat{h}: \IR^d \to \IR^+$ that is $C^{\infty}$ with compact support $[0,1]^d$ and so that there exists a function $h: \IR^d \to \IR$ which satisfies $h(0) >0$,
\[
|h(\x)| \leq \frac{A}{ 1+|\x|^{d+\tau} }  \hspace{2cm} \tau >0, 
\]
with the notation of \eqref{eq: controleCov}, and
\[
h(\x) = \int_{\IR^d} \hat{h}(\f)  e^{ \imath \f \cdot \x } d\f.
\]

It can be shown, similarly as in the proof of Lemma D.2 in \cite{bachoc14asymptotic}, that there exists a fixed $0 < \delta < \infty$  so that 
\begin{equation} \label{eq: in_proof_minoration_eigen_values_2} 
\sup_{1 \leq k \leq p, n_k \in \IN^+,1 \leq i \leq n_k} \,\,\sum_{1 \leq j \leq n_k; j \neq i} \,\left| h\left\{ \delta \big( \x^{(k)}_i - \x^{(k)}_j \big)\right\}  \right| \leq  \frac{1}{2} h\left(0\right).
\end{equation}

For all $1 \leq k \leq p$,
using Gershgorin circle theorem, the eigenvalues of
the $n_k \times n_k$ symmetric matrix $\big[ h\big\{ \delta \big( \x^{(k)}_i - \x^{(k)}_j \big) \big\} \big]_{1 \leq i,j \leq n_k}$ belong to the balls with center $h(0)$ and radius 
$\sum_{1 \leq j \leq n_k, j \neq i} \big| h\big\{ \delta \big( \x^{(k)}_i - \x^{(k)}_j \big) \big\} \big|$. Thus, because of \eqref{eq: in_proof_minoration_eigen_values_2}, these eigenvalues belong to the segment $[ h(0) - (1/2) h(0) , h(0) + (1/2) h(0)]$ and are larger than $(1/2) h(0)$. Hence, the $N \times N $ matrix $\bfT$ defined as being block diagonal, with $p$ blocks and with block $k$ equal to $ \big[ h\big\{ \delta \big( \x^{(k)}_i - \x^{(k)}_j \big)\big\} \big]_{1 \leq i,j \leq n_k}$, also has eigenvalues larger than $(1/2)h(0)$.

Consider now a real vector $\v$ of size $N \times 1$. Then, because $\bfT$ has eigenvalues larger than $(1/2) h(0)$,
\begin{equation} \label{eq:instrumental:large:enough}
 \frac{1}{2} h\left(0\right) \sum_{k=1}^p \sum_{i=1}^{n_k} | v_{N_{k-1} + i} |^2  \leq \sum_{k=1}^p \sum_{i,j=1}^{n_k}  v_{N_{k-1} + i} v_{N_{k-1} + j} h\left\{ \delta  \big( \x^{(k)}_i - \x^{(k)}_j \big)\right\} .
 \end{equation}

Now, define the $Np \times 1$ real vector $\w$ as follows: The $N$ first components are
$$
v_1,\dots,v_{n_1},\underbrace{0,\dots,0,}_{\mbox{$N - n_1$ times}}
$$
the components $N+1$ to $2N$ are
$$
\underbrace{0,\dots,0,}_{\mbox{$n_1$ times}}~v_{n_1+1},\dots,v_{n_1+n_2},\underbrace{0,\dots,0,}_{\mbox{$N - n_1 - n_2$ times}}
$$
and so on, until the $N$ last components are
$$
\underbrace{0,\dots,0,}_{\mbox{$N_{p-1}$ times}}\,v_{N_{p-1}+1},\dots,v_{N_{p-1}+n_p}.
$$
Let also $\{ \s_1,\dots,\s_N\}$ be defined as $\{ \x^{(1)}_1,\dots,\x^{(1)}_{n_1},\x^{(2)}_1,\dots,\x^{(2)}_{n_2},\dots,\x^{(p)}_1,\dots,\x^{(p)}_{n_p} \}$. Then we have, from \eqref{eq:instrumental:large:enough},
$$
\frac{1}{2} h\left(0\right) \sum_{k=1}^p \sum_{i=1}^{n_k} | v_{N_{k-1} + i} |^2  \leq
\sum_{k=1}^p \sum_{i,j=1}^N w_{(k-1)N + i} w_{(k-1)N + j} h\left\{ \delta  ( \s_i - \s_j )\right\} .
$$
Let, for $1 \leq i \leq N$, $\w_{i \bullet}$ be the $p \times 1$ vector $ (w_{(k-1)N+i})_{1 \leq k \leq p} $. Then, we can write
\begin{eqnarray} \label{eq:instrumental}
 \frac{1}{2} h\left(0\right) \sum_{k=1}^p \sum_{i=1}^{n_k} | v_{N_{k-1} + i} |^2 & \leq & \sum_{i,j=1}^N \sum_{k=1}^p w_{(k-1)N + i} w_{(k-1)N + j} \frac{1}{\delta^d} \int_{\IR^d} \hat{h}\left(\frac{\f}{\delta}\right) e^{ \imath \f \cdot ( \s_i - \s_j ) } d \f \nonumber \\
& = &  \frac{1}{\delta^d} \int_{\IR^d} \hat{h}\left(\frac{\f}{\delta}\right) \sum_{i,j=1}^N
( e^{ - \imath \f \cdot \s_i } \w_{i \bullet} )^* ( e^{ - \imath \f \cdot \s_j } \w_{j \bullet} )  d \f \nonumber \\
& = &  \frac{1}{\delta^d} \int_{\IR^d} \hat{h}\left(\frac{\f}{\delta}\right) \left| \left| \sum_{i=1}^N  e^{ - \imath \f \cdot \s_i } \w_{i \bullet}  \right| \right|^2 d \f.
\end{eqnarray}
Let $E$ be the set of $\f$ so that $\hat{h}\left(\f / \delta \right)$ is non-zero. Then $E$ is bounded since $\hat{h}$ has a compact support. Hence, because the Hermitian matrix $\hat{\C}(\f)$ has strictly positive eigenvalues for all $\f \in \IR^d$ and is continuous, we have $\inf_{\f \in E} \lambda_1\{ \hat{\C}(\f)\} >0$. Also, because $\hat{h}$ is continuous, $\sup_{\f \in E} \hat{h} ( \f / \delta ) < + \infty$. Hence, there exists a fixed $\delta_2 >0$ so that for any $\z \in \IC^{p}$, $\f \in \IR^d$, $ \z^* \hat{\C}(\f) \z \geq \delta_2 \hat{h}\left(\f / \delta \right) || \z ||^2$. Hence, from \eqref{eq:instrumental},
\begin{eqnarray*}
 \frac{1}{2} h\left(0\right) \sum_{k=1}^p \sum_{i=1}^{n_k} | v_{N_{k-1} + i} |^2 & \leq &  \frac{1}{\delta^d \delta_2} \int_{\IR^d} \left( \sum_{i=1}^N e^{ - \imath \f \cdot \s_i } \w_{i \bullet}  \right)^* \hat{\C} (\f) \left( \sum_{i=1}^N e^{ - \imath \f \cdot \s_i } \w_{i \bullet}  \right) d \f \\
& = & \frac{1}{\delta^d \delta_2} \sum_{i,j=1}^N \sum_{k,\ell=1}^p w_{(k-1)N+i} w_{(\ell-1)N+j}  \int_{\IR^d} \hat{c}_{k\ell}(\f) e^{ \imath \f \cdot ( \s_i - \s_j) } d \f  \\
& = & \frac{1}{\delta^d \delta_2} \sum_{k,\ell=1}^p \sum_{i,j=1}^N w_{(k-1)N+i} w_{(\ell-1)N + j}
c_{k\ell}(\s_i - \s_j) \\
& = & \frac{1}{\delta^d \delta_2} \v^t \bSigma \v.
\end{eqnarray*}
This concludes the proof.
\end{proof}

We now extend Theorem \ref{thm} to the case of a parametric family of covariance and cross covariance functions $\{c_{\btheta k\ell}(\h)\}_{1 \leq k,\ell \leq p}$, indexed by a parameter $\btheta$ in a compact set $\Theta$ of $\IR^q$. Let $\hat{\C}_{\btheta}(\f)$ be as $\hat{\C}(\f)$ with $\{c_{k\ell}(\h)\}_{1 \leq k,\ell \leq p}$ replaced by $\{c_{\btheta k\ell}(\h)\}_{1 \leq k,\ell \leq p}$.
The proof of the next theorem is identical to that of Theorem \ref{thm}, up to more cumbersome notations, and is omitted.

\begin{thm} \label{thm:deux}
Assume that, for all $\btheta \in \Theta$, the functions $\{c_{\btheta k\ell}(\h)\}_{1 \leq k,\ell \leq p}$ satisfy Assumption \ref{ass:decay:infinity}, where $A$ and $\tau$ can be chosen independently of $\btheta$.
Assume that \eqref{eq:inverse:fourier} holds with $\{c_{k\ell}(\h)\}_{1 \leq k,\ell \leq p}$ replaced by $\{c_{\btheta k\ell}(\h)\}_{1 \leq k,\ell \leq p}$, for all $\btheta \in \Theta$. Assume that $\hat{\C}_{\btheta}(\f)$ is jointly continuous in $\f$ and $\btheta$ and that $\lambda_1\{\hat{\C}_{\btheta}(\f)\} >0$ for all $\f$ and $\btheta$. Assume finally that Assumption \ref{ass:Delta} is satisfied.

Then, with $\bSigma_{\btheta}$ being as $\bSigma$ with $\{c_{k\ell}(\h)\}_{1 \leq k,\ell \leq p}$ replaced by $\{c_{\btheta k\ell}(\h)\}_{1 \leq k,\ell \leq p}$, we have
\[
\inf_{\btheta \in \Theta,n_1,\dots,n_p \in \IN^+} \lambda_1(\bSigma_{\btheta}) > 0.
\]
\end{thm}

The lower bound provided by Theorem \ref{thm:deux} is typically assumed in the references \cite{shaby12tapered}, \cite{Beli:etal:15} and \cite{furrer15asymptotic}.

\section{Concluding remarks} \label{section:conclusion}

Note that Theorem \ref{thm} is typically applicable for models of stochastic processes with discrete definition spaces, such as time series or Gauss-Markov random fields. Indeed, these models incorporate a fixed minimal distance between any two different observation points.

It is well-known (e.g. \citealp{wackernagel03multivariate}) that covariance and cross covariance functions $\{c_{k\ell}(\h)\}_{1 \leq k,\ell \leq p}$ satisfy $\lambda_1\{\hat{\C}(\f)\} \geq 0$ for all $\f$. It is also known that if $\lambda_1\{\hat{\C}(\f)\} > 0$ for almost all $\f \in \IR^d$, then $\lambda_1(\bSigma) > 0$ whenever the points $(\x^{(k)}_i)_{1 \leq i \leq n_k}$ are two-by-two distinct for all $k$. We believe that this latter assumption on $\lambda_1\{\hat{\C}(\f)\}$ is nevertheless generally insufficient for Theorem \ref{thm} to hold. Indeed, consider for illustration the univariate and unidimensional case with the triangular covariance function, that is $d=1$, $p=1$ and $c(h) = c_{11}(h) = (1 - |h|) \1_{\{ |h| \leq 1 \}}$. Then (see, e.g., \citealp{stein99interpolation}), $\hat{c}(f) = \sinc^2( f / 2)/(2 \pi )$ has a countable number of zeros.
Consider now the sequence of observation points $(x^{(1)}_i)_{i \in \IN^+} = (x_i)_{i \in \IN^+} = (i/2)_{i \in \IN^+}$. Then, $\bSigma$ is a tridiagonal Toeplitz matrix with $\sigma_{11}=1$ and $\sigma_{12} = \sigma_{21} = 1/2$. Hence, see e.g., \cite{noschese13tridiagonal}, for any value of $n_1 = n$, the eigenvalues of $\bSigma$ are $[1 + \cos\{(i \pi)/(n+1)\}]_{1 \leq i \leq n}$. Thus, although $\lambda_1(\bSigma) = 1 + cos\{(n \pi)/(n+1) \}$ is strictly positive for any $n$, it goes to $0$ as $n \to \infty$. Hence, Assumption \ref{ass:fourier}, stating that $\lambda_1\{\hat{\C}(\f)\}$ is strictly positive for all $\f$, appears to be generally necessary for Theorem~\ref{thm} to hold.

In order to derive increasing-domain asymptotic properties for covariance tapering (e.g. in \citealp{shaby12tapered} or \citealp{Beli:etal:15}), one typically assumes that the smallest eigenvalues of tapered covariance matrices are lower bounded, uniformly in $n_1,\dots,n_p$. A tapered covariance matrix is of the form $\bSigma \circ \A $, where $\circ$ is the Schur (component by component) product and where $\A$ is symmetric positive semi-definite with diagonal components $1$. Because of the relation $\lambda_1( \bSigma \circ \A ) \geq ( \min_{i} A_{ii} ) \lambda_1( \bSigma )$ (see \citealp[Theorem 5.3.4]{horn91topics}), Theorem~\ref{thm} directly provides a uniform lower bound for the smallest eigenvalues of tapered covariance matrices.

~ \\

{\bf Acknowledgment:}Reinhard Furrer  acknowledges support of the UZH Research Priority
Program (URPP) on ``Global Change and Biodiversity'' and the Swiss National Science
Foundation SNSF-143282.

\bibliographystyle{alpha}
\bibliography{bibtex}

\end{document}